\documentclass[11pt]{amsart}
\usepackage{amsfonts,amssymb,amsmath,amsthm}
\usepackage{url}
\usepackage{enumerate}
\usepackage{bbm}
\usepackage[all]{xy}
\usepackage[pdftex, colorlinks, citecolor=red, backref=page]{hyperref}
\usepackage{color,xcolor}

\newtheorem{theorem}{Theorem}[section]

\newtheorem{definition}[theorem]{Definition}
\newtheorem{notion}[theorem]{}

\newtheorem{proposition}[theorem]{Proposition}
\newtheorem{corollary}[theorem]{Corollary}
\theoremstyle{definition}
\newtheorem{remark}[theorem]{Remark}

\newtheorem{example}[theorem]{Example}

\author{Qingnan An}
\address{School of Mathematics and Statistics, Northeast Normal University, Changchun, {\rm130024}, China}
\email{qingnanan1024@outlook.com}

\author{S\o ren Eilers}
\address{Department of Mathematical Sciences, University of Copenhagen, Copenhagen, Denmark}
\email{eilers@math.ku.dk}

\author{Guihua Gong}
\address{Department of Mathematics, Hebei Normal University, Shijiazhuang, Hebei {\rm050016}, China
and Department of Mathematics, University of Puerto Rico, Rio Piedras, PR 00936, USA}
\email{ghgong@gmail.com}

\author{Zhichao Liu}
\address{School of Mathematical Sciences,
Dalian University of Technology,
Dalian, {\rm116024}, China}
\email{lzc.12@outlook.com}


\begin{document}

\title[ASH algebras]
{Subhomogeneity in the classification of real rank zero ${\rm C}^*$-algebras}

\begin{abstract}
In this paper, we construct a class of ASH algebras of real rank zero and stable rank one which is not K-pure. Then we show the following:

(i) There exists a real rank zero inductive limit of 1-dimensional noncommutative CW complexes which is not an A$\mathcal{HD}$ algebra, when $K_1$ is torsion free or has bounded torsion.

(ii) Total K-theory is not a complete invariant for ASH algebras of real rank zero.

(iii) There are obstructions both in the total K-theory of ideals and quotients in the classification of $C^*$-algebras of real rank zero and stable rank one. 





\end{abstract}

\maketitle
\section{Introduction}
The Elliott classification programme has achieved significant success in the classification of $C^*$-algebras. This programme started with the classification of AT-algebras of real rank zero, in which Elliott classified these algebras in terms of the ordered $\mathrm{K}_{*}$-group \cite{Ell,EG}. Examples exhibited by Gong \cite{G} showed that  the ordered $\mathrm{K}_{*}$-group is not sufficient for the classification of (non-simple) real rank zero
approximate homogeneous $C^*$-algebras (AH algebras). Elliott, Gong, and Su \cite{EGS} constructed similar examples for AD algebras by employing AH algebras (indirectly).  Dadarlat and Loring \cite{DL3} also gave such an example for AD algebras directly. The effort of many mathematicians \cite{DG,DL2,DL3,Ei} led to a refined invariant, the total $\mathrm{K}$-theory group
$$
\underline{\mathrm{K}}(A)=\bigoplus_{n \geq 0} \mathrm{K}_{*}(A ; \mathbb{Z}_n)
$$
endowed with a certain order structure and acted upon by the natural coefficient transformations $\rho_n^i: \mathrm{K}_i(A) \rightarrow \mathrm{K}_i(A ; \mathbb{Z}_n),$
$ \kappa_{m, n}^i:  \mathrm{K}_i(A ; \mathbb{Z}_n) \rightarrow \mathrm{K}_i(A ; \mathbb{Z}_m)$ and by the Bockstein maps  $\beta_n^i: \mathrm{K}_i(A ; \mathbb{Z}_n) \rightarrow \mathrm{K}_{i+1}(A)$ studied in \cite{RS,S}.
The collection of all the maps $\rho, \beta, \kappa$ and their compositions are called Bockstein operations and  denoted by $\Lambda$.

The total K-theory $\underline{\mathrm{K}}(A)$ provides a better description of the connections between ideals of $A$. Eilers \cite{Ei} showed that $\underline{\mathrm{K}}(A)$ is a complete invariant for AD algebras with real rank zero and bounded torsion in $\mathrm{K}_1$. Later, Dadalat and Loring \cite{DL3} generalized this classification to all real rank zero AD algebras. In 1997, Dadarlat and Gong \cite{DG} presented an important classification of certain approximately subhomogeneous ${C}^*$-algebras of real rank zero with no dimension growth. (In \cite[section 1]{DG}, this class is called the ASH(2) class; in the subsequent work, this class is called the  A$\mathcal{HD}$ class, because the class includes all real rank zero AD algebras and AH algebras with slow dimension growth.) Recently, \cite{AELL, ALZ} push forward a classification for certain inductive limit of 1-NCCW complexes (it includes much more general subhomogeneous building blocks than \cite{DL3}).

Since the total K-theory can be defined using dimension drops algebras exclusively, it seems reasonable that these form a class inside 1-NCCW complexes which somehow generates everything. There are some results indicating these reductions analogous to AH algebras \cite{GJLP1}.
In 1999, Eilers, Loring and Pedersen \cite{ELP} showed that $A$ is an inductive limit of 1-NCCW complexes if and only if $A$ is local 1-NCCW complex, moreover, if $A$ has real rank zero and ${\rm K}_1(A)=0$, then $A$ is an AF algebra. But if ${\rm K}_1(A)$ is torsion free and non-zero, it is likely that $A$ is an AT algebra. Then there  arises the natural question of whether any real rank zero inductive limit of 1-NCCW complexes are AD algebras. Integrated with AH algebras, it is quite possible that the A$\mathcal{HD}$ class could cover all real rank zero ASH algebras with no dimension growth.

In this paper, we  give a negative answer to all these questions.
\begin{theorem} {\rm (Theorem \ref{ex1} and  Example \ref{ex torsion})}
There exists a real rank zero $C^*$-algebra $E$ which can be written as an inductive limit of 1-NCCW complexes but $E$ is not an A$\mathcal{HD}$ algebra. Such $E$ can have torsion free or bounded torsion $K_1$-group.
\end{theorem}

Our strategy is to construct non-K-pure real rank zero inductive limits of 1-NCCW complexes. We find there are many 1-NCCW complexes with ``odd" properties (Definition \ref{meiyuchou}) and we can use these odd blocks to obtain non-K-pure inductive limits. The K-pureness is a useful property to characterize the order structure of $C^*$-algebras. This property indicates the extent to which the total K-theory of ideals can be preserved through embeddings. In the non-K-pure case, there exists two $C^*$-algebras of real rank zero and stable rank one with the same total K-theory, but they have different total Cuntz semigroup \cite{ALcounter}. The key point is that the isomorphism between the total K-theory preserves all the Bockstein maps except the Dadalat-Gong order for ideals. It is believed that all ASH algebras of real rank zero can be classified by the total K-theory. We will use these ``odd" blocks to construct a counterexample in the ASH class.
\begin{theorem}  {\rm (Example \ref{ex2} and  Example \ref{ex Z})}
There exist two real rank zero $C^*$-algebras $E_1,E_2$ which can be written as  inductive limits of 1-NCCW complexes such that they have the same total K-theory (total Cuntz semigroup) but they are not isomorphic. These $E_1,E_2$ can be chosen to be 
$\mathcal{Z}$-stable.
\end{theorem}

  In our construction, $E_1$ and $E_2$ have isomorphic ideals and ideal lattice structure, but their quotients are non-isomorphic. It is shown in \cite{ALcounter} that there exist obstructions in the total K-theory of ideals,  the total Cuntz semigroup would catch all the information of ideals.
  Our result will indicate that there exist obstructions in the total K-theory of quotients when the quotient is not simple. 
  Combining with the total K-theory of ideals and  quotients, we are able to give a classification.
  \begin{theorem} {\rm (Theorem \ref{mainthm})}
Given A$\mathcal{HD}$ algebras $A_i,B_i,$ $i=1,2$  of stable rank one and real rank zero with $A_i$ unital and $B_i$ stable and
 unital full extensions with trivial boundary maps:
$$
0\to B_i\to E_i\to A_i\to 0,\quad i=1,2.
$$
If the following diagram is commutative in $\Lambda$-category
$$\xymatrixcolsep{2pc}
\xymatrix{
 {\,\,\underline{\mathrm{K}}(B_1)\,\,} \ar[d]_-{\cong_+}\ar[r]^-{\underline{\mathrm{K}}(\iota_1)}
& {\,\,(\underline{\mathrm{K}}(E_1),[1_{E_1}])\,\,} \ar[d]_-{\cong_+} \ar[r]^-{\underline{\mathrm{K}}(\pi_1)}
& {\,\,(\underline{\mathrm{K}}(A_1),[1_{A_1}])\,\,}\ar[d]_-{\cong_+}\\
{\,\,\underline{\mathrm{K}}(B_2)\,\,} \ar[r]_-{\underline{\mathrm{K}}(\iota_2)}
& {\,\,(\underline{\mathrm{K}}(E_2),[1_{E_2}]) \,\,} \ar[r]_-{\underline{\mathrm{K}}(\pi_2)}
& {\,\,(\underline{\mathrm{K}}(A_2),[1_{A_2}]) \,\,},}
$$
where $\cong_+$ means the ordered scaled isomorphism, then
 $E_1\cong E_2.$
  \end{theorem}

 We remark that  if the $C^*$-algebra is K-pure, these obstructions from ideals and quotients will vanish. It is expected that any K-pure real rank zero inductive limits of 1-dimensional NCCW complexes are AD algebras.

\section{Preliminaries}

\begin{definition}\label{ET}\rm
Let $F_1$ and $F_2$ be two finite dimensional ${C}^*$-algebras and let
$\varphi_0,\,\varphi_1:\,F_1\to F_2$ be two homomorphisms.
Set
\begin{align*}
A&=A(F_1,F_2,\varphi_0,\varphi_1)
\\
&=\{(f,a)\in  C([0,1],F_2) \oplus F_1:\,f(0)=\varphi_0(a)\,  {\rm and}\, f(1)=\varphi_1(a)\}.
\end{align*}
These ${\mathrm C}^*$-algebras are called 1-dimensional non-commutative finite CW (1-NCCW) complexes (see \cite{Ell:1996, ET:1994}). So all the 1-NCCW complexes form a subclass of subhomogenous ${\mathrm C}^*$-algebras.
\end{definition}
\begin{proposition} \label{k nccw} {\rm (}\cite{GLN1}{\rm )}
  Let $A=A(F_1,F_2,\varphi_0,\varphi_1)$ 
  and let $\varphi_{0*},\varphi_{1*}:{\rm K}_0(F_1)=\mathbb{Z}^p\rightarrow {\rm K}_0(F_2)=\mathbb{Z}^l$ be represented by matrices $\alpha=({\alpha_{ij}})_{l\times p}$ and $\beta=({\beta_{ij}})_{l\times p}$, where $\alpha_{ij},\beta_{ij}\in \mathbb{Z}_+$ for each pair $i,j$.
  Then
  $$
  {\rm K}_0(A)={\rm Ker}(\alpha -\beta),
\quad
{\rm K}_0^+(A)={\rm Ker}(\alpha -\beta)\cap \mathbb{Z}_+^p,
\quad
 {\rm K}_1(A)=\mathbb{Z}^l/{\rm Im}(\alpha-\beta).
 $$

\end{proposition}
\begin{definition} \rm
Suppose that
 $F_1=\bigoplus_{j=1}^p M_{k_j}(\mathbb{C})$, $F_2$ $=\bigoplus_{i=1}^lM_{h_i}(\mathbb{C})$.
  Let us use $\theta_1,\theta_2,\cdots,\theta_p$ to denote the spectrum of $F_1$ and denote the spectrum of $C([0,1],F_2)$ by $(t,i)$,
 where $0\leq t\leq1$ and $i\in \{1,2,\cdots,l\}$ indicates that it is in $i^{th}$ block of $F_2$. So
  $$
  {\rm Prim}(A)=\{\theta_1,\theta_2,\cdots,\theta_p\}\cup \coprod_{i=1}^{l}(0,1)_i
  $$
  the topology base at each point $\theta_j$ is given by
  $$
  \{\theta_j\}\cup\coprod_{\{i|\alpha_{ij}\neq0\}}(0,\varepsilon)_i\cup\coprod_{\{i|\beta_{ij}\neq0\}}(1-\varepsilon,1)_i.
  $$
  In general, this is a non Hausdorff topology.

 Writing $a\in F_1$ as $a=(a(\theta_1),a(\theta_2)$, $\cdots,$ $a(\theta_p))$,
  $f(t)\in C([0,1],F_2)$ as $$f(t)=(f(t,1),f(t,2),\cdots,f(t,l))$$ where $a(\theta_j)\in M_{k_j}(\mathbb{C})$,
  $f(t,i)\in C([0,1],M_{h_i}(\mathbb{C}))$.

  For any $(f,a)\in A$, 
  there is a canonical map $\pi_e:\,A\rightarrow F_1$ defined by $\pi_e((f,a))=a$.
\end{definition}

\begin{definition}\rm
 We shall say a ${C}^*$-algebra is an A$\mathcal{HD}$ algebra (see \cite[2.2-2.3]{GJL} and  \cite{GJL2}) if it is an inductive limit of finite direct sums of algebras $M_n(\widetilde{\mathbb{I}}_p)$ and $PM_n(C(X))P$, where $$
\mathbb{I}_p=\{f\in M_p(C_0(0,1]):\,f(1)=\lambda\cdot1_p,\,1_p {\rm \,is\, the\, identity\, of}\, M_p\}
$$
and $X$ is one of the following finite connected CW complexes: $\{pt\},~\mathbb{T},~[0, 1],~T_{II,k}.$ Here, $P\in M_n(C(X))$ is a projection and $T_{II,k}$ is the 2-dimensional connected simplicial complex with $H^1(T_{II,k})=0$ and $H^2(T_{II,k})=\mathbb{Z}/k\mathbb{Z}$. (This class of ${\rm C}^*$-algebras also plays an important role in the classification of ${\rm C}^*$-algebras with the  ideal property \cite{GJLP1,GJLP2}.)
\end{definition}

\begin{definition}\rm  \label{def kpure}
Let $A$ and $B$ be ${C}^*$-algebras. An extension $e$ of $A$ by $B$ is a short exact sequence of ${\rm C}^*$-algebras:
 $$
e:\quad 0 \rightarrow B \rightarrow E \rightarrow A \rightarrow 0.
$$
We say $e$ is $stenotic$ if $I\subset B$ or $B\subset I$ for any ideal $I$ of $E$.  We say $e$ is ${K}$-$pure$,
if both the sequences
$$0\to  {\rm K}_j(B ) \to {\rm K}_j(E)\to {\rm K}_j(A ) \to 0,\quad j=0,1$$
are pure extensions of abelian groups.

We say a ${\rm C}^*$-algebra $E$ is $K$-$pure$, if for any ideal $I$ of $E$, the extension
$
0\to I\to E\to E/I\to 0
$
is ${K}$-pure.
\end{definition}

\begin{definition}\label{trivial boundary def}\rm
We say an extension $e:0\to B\to E\to A\to 0$ has $trivial$ $boundary$ $maps$, if both  $\delta_0$ and $\delta_1$ in the six-term exact sequence are the zero map.

Thus, suppose both $A,B$ are of stable rank one and real
rank zero; by \cite[Proposition 4]{LR}, $E$ has stable rank one and real
rank zero if, and only if, $e$ has trivial boundary maps.
\end{definition}

\begin{definition}\rm
Let $A$, $B$ be ${ C}^*$-algebras and let $e$
be an extension of $A$ by $B$ with Busby invariant $\tau:\,A\to M(B)/B$, where $M(B)$ is the multiplier algebra of $B$.
We say the extension $e$ is $full$, if $\tau$ is full.
When $A$ is unital, we say an extension $e$ is $unital$, if $\tau$ is unital.
\end{definition}

\begin{definition}\rm \label{strong ue}
Let
$e_i:0 \to B \to E_i \to A \to 0 $
be two
extensions of $A$ by $B$ with Busby invariants $\tau_i$ for $i = 1, 2$. We say $e_1$ and $e_2$ are  $strongly$ $unitarily$ $equivalent$, denoted by $e_1
\sim_s e_2$,
if there exists a unitary $u \in M(B)$ such that $\tau_2(a) = \pi(u)\tau_1(a)\pi(u)^*$ for all $a \in A$,
where $\pi:\,M(B)\to M(B)/B.$ Two unital extensions $e_1$ and $e_2$ are called $stably$ $strongly$ $unitarily$ $equivalent$, denoted by $e_1
\sim_{ss} e_2$, if there exist unital trivial extensions $\sigma_1,\sigma_2$  such that $e_1\oplus\sigma_1 \sim_s e_2\oplus\sigma_2$.

If $A$ is unital,  denote  $\mathrm{Ext}^u_{s}(A, B)$ the set of strongly unitary equivalence classes of unital extensions of $A$ by $B$
and  denote  $\mathrm{Ext}^u_{ss}(A, B)$ the set of stably strongly unitary equivalence classes of unital extensions of $A$ by $B$.


\end{definition}

We will also need the following \cite[Theorem 4.9]{W}; one can also see analogous versions in the circumstances of \cite{GR,Sk1,Sk2,W0}.

\begin{theorem}{\rm (}\cite[Theorem 4.9]{W}{\rm )}\label{strong wei}
Suppose that $A$ is a unital separable nuclear ${C}^*$-algebra with $A\in\mathcal{N}$ and $B$ is separable stable ${C}^*$-algebra and has a countable approximate unit of projections. Then
there is a short exact sequence of groups
$$
0 \to \mathrm{Ext}_{[1]}(\mathrm{K}_*(A), \mathrm{K}_*(B)) \to \mathrm{Ext}_{ss}^u(A, B) \to {\rm Hom}_{[1]}(\mathrm{K}_*(A), \mathrm{K}_*(B)) \to 0.
$$
\end{theorem}

\begin{proposition}{\rm (}\cite[Lemma 3.3]{{ALpams}}{\rm )}\label{AL CFP}
Suppose that $B$ is a separable stable $C^*$-algebra with real rank zero, stable rank one and $({\rm K}_0(B),{\rm K}_0^+(B))$ is weakly unperforated $($in the sense of \cite{Ell0}$)$. Then $B$ has the corona factorization property.
\end{proposition}

The corona factorization property implies the following result; see \cite[Theorem 1.4]{KN2} or  Ng's survey paper \cite{Ng}.

\begin{proposition}{\rm (}\cite[Theorem 3.4]{ALpams}{\rm )}\label{AL full}
Suppose that $B$ is a separable stable $C^*$-algebra with real rank zero, stable rank one and $({\rm K}_0(B),{\rm K}_0^+(B))$  is weakly unperforated. Let $A$ be a unital separable $C^*$-algebra. Suppose that  either $A$ or $B$ is nuclear.  Then

{\rm (1)} every unital full extension of $A$ by $B$ is absorbing.

{\rm (2)} if $A$ is simple, then every unital extension of $A$ by $B$ is absorbing.

\end{proposition}

\section{The ASH construction}
 In this section, we will show that any real rank zero inductive limit of nice  1-NCCW complexes 
 is K-pure. For more general case, there exist non-K-pure algebras which are the real rank zero inductive limits.
\begin{definition}\label{meiyuchou}\rm
Let $A$ be an 1-NCCW complex. Let us say $A$ is $nice$ if for any compact ideal (ideal generated by projections) $I$ of $A$, the extension $0\rightarrow I\to A\to A/I\to 0$  is K-pure. If there exists a compact ideal $I$ of $A$ such that the extension $0\rightarrow I\to A\to A/I\to 0$ has trivial boundary maps but is not K-pure, we say $A$ is $odd$.
\end{definition}
 It is obviously that any direct sums of nice 1-NCCW complexes is still nice. Then we have the following proof similar with  \cite[Proposition 4.4]{DE}.
\begin{theorem}\label{purelimit}
 Every real rank zero inductive limit of direct sums of nice 1-NCCW complexes is K-pure.
\end{theorem}
\begin{proof}
 Write $A=\lim\limits_{\longrightarrow}\left(A_n, v_n\right)$ and let $v_{\infty, n}: A_n \rightarrow A$ denote the canonical map. Since $A$ has real rank zero, then $I$ and $A/I$ have real rank zero. Moreover, there is an increasing sequence of projections $\left(f_n\right)$ in $I$, forming an approximate unit of $I$. By induction, there exists a sequence of projections $e_n \in A_n$ such that $v_{\infty, n}\left(e_n\right)=f_n$ and $v_n\left(e_n\right) \leqslant e_{n+1}$. 

 Let $I_n$ denote the closed two-sided ideal of $A_n$ generated by $e_n$. Then $v_n\left(I_n\right) \subset I_{n+1}$ 
 and $I=\lim\limits_{\longrightarrow}\left(I_n, v_n\right)$.
 Now the given extension of $C^*$-algebras is the inductive limit of the sequence of extensions
$$
0 \rightarrow I_n \rightarrow A_n \rightarrow A_n / I_n \rightarrow 0,
$$
the bonding morphisms being induced by $\left(v_n\right)$.

Since $A_n$ is a nice 1-NCCW complex, then we have the sequence of $K_*$-groups
$$
0 \rightarrow {\rm K}_*(I_n) \rightarrow {\rm K}_*(A_n) \rightarrow {\rm K}_*(A_n / I_n) \rightarrow 0
$$
is pure exact. By the continuity of K-theory, the
short sequence of groups
$$
0 \rightarrow {\rm K}_*(I) \rightarrow {\rm K}_*(A) \rightarrow {\rm K}_*(A / I) \rightarrow 0
$$
is pure exact. Then $A$ is K-pure.
\end{proof}
\begin{remark}
  Let $A=A(F_1,F_2, \varphi_1,\varphi_2)$ be a 1-NCCW complex. If $F_2$ is simple, for any compact ideal $I$ of $A$, $A/I$ is either $A$ or a finite dimensional $C^*$-algebra. That is, the extension $0\rightarrow I\to A\to A/I\to 0$ is K-pure, hence, $A$ is nice. Then by Theorem \ref{purelimit}, one can obtain that the algebras classified in \cite{ALZ} by total K-theory are K-pure.
\end{remark}
\begin{theorem}\label{ex1}
There exists a unital, separable, nuclear, $C^*$-algebra $E$ of real rank zero and stable rank one, which is not K-pure but satisfies the following:

{\rm (i)} $E$ is an inductive limit of 1-NCCW complexes;

{\rm (ii)} $E$ forms a  unital full non-K-pure extension of simple AD algebras;

{\rm (iii)} $E$ is not an A$\mathcal{HD}$ algebra;

{\rm (iv)} ${\rm K}_*(E)$ is not weakly unperforated.
\end{theorem}
\begin{proof}
Let
$$
F_{1,n}=M_{3^n}\oplus M_{3^n} \oplus M_{3^n},\,\,\,\,F_{2,n}=M_{2\cdot 3^n}\oplus M_{2\cdot 3^n}.
$$
 Define  $\varphi_{0,n},\,\varphi_{1,n}:\,F_{1,n}\to F_{2,n}$ as
$$
\varphi_{0,n} (a)=
\left(\begin{array}{cc}
        a(\theta_1) &  \\
         & a(\theta_1)
      \end{array}
\right)\oplus \left(\begin{array}{cc}
        a(\theta_1) &  \\
         & a(\theta_3)
      \end{array}
\right),$$
$$
\varphi_{1,n} (a)=
\left(\begin{array}{cc}
        a(\theta_2) &  \\
         & a(\theta_2)
      \end{array}
\right)\oplus \left(\begin{array}{cc}
        a(\theta_2) &  \\
         & a(\theta_3)
      \end{array}
\right).
$$
Set
$$
C_n=C_n(F_{1,n},F_{2,n},\varphi_{0,n},\varphi_{1,n}).
$$
Then for any $n$,  $C_n\cong M_{3^n}(C_0)$ and
$$
{\rm Prim}(C_n)=\{\theta_{1}, \theta_{2},\theta_{3}\}\cup\coprod_{i=1}^{2}(0,1)_i.$$ By Proposition \ref{k nccw}, we have
$$
{\rm K}_0(C_n)=\mathbb{Z}\oplus\mathbb{Z},\quad {\rm K}_1(C_n)=\mathbb{Z}^2/{\small{\rm Im} \left(\begin{array}{ccc}
                                                     2 & -2 &0\\
                                                     1 & -1&0
                                                   \end{array}\right)} \cong \mathbb{Z},
$$
where the generator of ${\rm K}_1(C_n)$ is 
$\{\scriptsize  \left(\begin{array}{c}
       1 \\
       1
     \end{array}\right)+k\left(\begin{array}{c}
       2 \\
       1
     \end{array}\right)\,|\, k\in \mathbb{Z}\}$.


Let $\{t_n\}$ be a sequence dense in $(0,1)$. Define $\psi_{n,n+1}: C_n\rightarrow C_{n+1}$ by
\begin{align*}
\psi_{n,n+1}(f,a)= & u_n\left(\begin{array}{ccc}
                 f(t,1) &  &   \\
                 & f(t_n,1) &  \\
                  & & f(t_n,1)  \\
               \end{array}
\right)u_n^* \\
   & \oplus v_n\left(\begin{array}{ccc}
                 f(t,2) &  &   \\
                 & f(t_n, 1) &  \\
                  & & f(t_n,2)  \\
      \end{array}
\right)v_n^*,
\end{align*}
where $u_n,v_n$ are unitaries in $M_{2\cdot 3^{n+1}}(\mathbb{C})$ such that

\begin{align*}
\pi_e\circ \psi_{n,n+1}(f,a)= & \left(\begin{array}{cc}
         a(\theta_{1}) &  \\
         & f(t_n,1)
      \end{array}
\right)\oplus \left(\begin{array}{cc}
         a(\theta_{2}) &  \\
         & f(t_n,1)
      \end{array}
\right)\\
   & \oplus \left(\begin{array}{cc}
         a(\theta_{3}) &  \\
         & f(t_n,2)
      \end{array}\right).
\end{align*}

Let  $E=\lim\limits_{\longrightarrow}(C_n,\phi_{n,n+1})$, then $E$ is a real rank zero inductive limit of 1-NCCW complexes.
Since every 1-NCCW complex has stable rank one,  then $E$ has stable rank one.

For any $n\in\mathbb{N}$, let
$$
p_n=\left(\begin{array}{cc}
        0_{3^n} &  \\
         & 0_{3^n}
      \end{array}
\right)\oplus \left(\begin{array}{cc}
        0_{3^n}&  \\
         & 1_{3^n}
      \end{array}
\right),
$$
then $p_n$ is a projection in $C_n$. Let  $I_n$ denote the ideal of $C_n$ generated by $p_n$, then
$$
I_n=\{f\in M_{2\cdot 3^n}(C[0,1])\mid f(0)=f(1)\in 0_{3^n}\oplus M_{3^n}\},
$$
and ${\rm K}_0(I_n)={\rm K}_1(I_n)=\mathbb{Z}$.

Consider the natrual embedding map  $I_n\rightarrow C_n$, the generator of ${\rm K}_1(I_n)$ corresponds to $\{\scriptsize \left(\begin{array}{c}
       0 \\
       1
     \end{array}\right)+k\left(\begin{array}{c}
       2 \\
       1
     \end{array}\right)\mid k\in \mathbb{Z}\}=\{\scriptsize \left(\begin{array}{c}
       2 \\
       2
     \end{array}\right)+k\left(\begin{array}{c}
       2 \\
       1
     \end{array}\right)\mid k\in \mathbb{Z}\}$, which is just the double of the generator of ${\rm K}_1(C_n)$! That is,
     the short exact sequence
     $$
     0\to {\rm K}_1(I_n)\to  {\rm K}_1(C_n)\to {\rm K}_1(M_{3^n}(\widetilde{\mathbb{I}}_2))\to 0    $$
     is exactly
     $$
0\to \mathbb{Z}\xrightarrow{\times 2} \mathbb{Z}\to \mathbb{Z}_2\to 0.
$$
Hence, $
     0\to I_n\to  C_n\to M_{3^n}(\widetilde{\mathbb{I}}_2)\to 0    $
isn't a K-pure extension and $C_n$ is an odd 1-NCCW complex (Definition \ref{meiyuchou}).

Note that $E$ has an ideal $I$  which can be written as the following:
$$
I_0\rightarrow I_1\rightarrow I_2\rightarrow\cdots\to I.
$$
The quotient algebra $E/I=\lim\limits_{\longrightarrow}(M_{3^n}(\widetilde{\mathbb{I}}_2),\phi^*_{n,n+1})$ is an inductive limit of dimension drop algebras,
where $\phi^*_{n,n+1}$ is induced by $\phi_{n,n+1}$. It is easily seen that both $I$ and $E/I$ are simple.

Note that the ${\rm K}_0$ of the connecting map $\psi_{n,n+1}: C_n\rightarrow C_{n+1}$ is ${\scriptsize \left(\begin{array}{cc}
   3 & 0 \\
   1 & 2
 \end{array}\right)}$, the ${\rm K}_1$ of the connecting map $\psi_{n,n+1}:C_n\rightarrow C_{n+1}$ is the identity map, then
$$
{\,\,{\rm K}_0(C_0)} \to {\,\,{\rm K}_0(C_1)} \to
 {\,\,{\rm K}_0(C_2)}  \to
 {\,\,\cdots}\to
 {\,\,\,{\rm K}_0(E)}.
$$
$$
{\scriptsize {\,\,\left(\begin{array}{c}
       1 \\
       0
     \end{array}\right)\,\,} \mapsto {\,\,\left(\begin{array}{c}
       3 \\
       1
     \end{array}\right)\,\,} \mapsto {\,\,\left(\begin{array}{c}
       9 \\
       5
     \end{array}\right) \,\,} \mapsto {\,\, \cdots\,\,}  \mapsto {\,\,\left(\begin{array}{c}
       1 \\
       0
     \end{array}\right)\,\,} }
$$
$$
{\scriptsize{\,\,\left(\begin{array}{c}
       0 \\
       1
     \end{array}\right)\,\,} \mapsto {\,\,\left(\begin{array}{c}
       0 \\
       2
     \end{array}\right)\,\,} \mapsto {\,\,\left(\begin{array}{c}
       0 \\
       4
     \end{array}\right) \,\,} \mapsto {\,\, \cdots\,\,}\mapsto {\,\,\left(\begin{array}{c}
       0 \\
       1
     \end{array}\right)\,\,}}
$$

Now we have
$$
{\rm K}_0(I)=\mathbb{Z}\left[\frac{1}{2}\right], \,\,{\rm K}_0(E)=\mathbb{Z}\left[\frac{1}{3}\right]\oplus \mathbb{Z}\left[\frac{1}{2}\right],\,\,{\rm K}_0(E/I)=\mathbb{Z}\left[\frac{1}{3}\right],
$$
$$
{\rm K}_1(I)=\mathbb{Z},\,\,{\rm K}_1(E)=\mathbb{Z},\,\,{\rm K}_1(E/I)=\mathbb{Z}_2.
$$

Notice that we don't have $$\left(\begin{array}{c}
       1 \\
       0
     \end{array}\right)\geq \left(\begin{array}{c}
       0 \\
       1
     \end{array}\right)\quad{\rm in}\quad {\rm K}_0(C_0),$$
      and we  also don't have
      $$
      \left(\begin{array}{c}
       3 \\
       1
     \end{array}\right)\geq \left(\begin{array}{c}
       0 \\
       2
     \end{array}\right)\,\,\quad{\rm in}\quad {\rm K}_0(C_1), $$
but we have
$$
\left(\begin{array}{c}
       9 \\
       5
     \end{array}\right)\geq \left(\begin{array}{c}
       0 \\
       4
     \end{array}\right)\quad {\rm in}\quad {\rm K}_0(C_2),$$
     this implies
     $$\left(\begin{array}{c}
       1 \\
       0
     \end{array}\right)\geq \left(\begin{array}{c}
       0 \\
       1
     \end{array}\right)\quad {\rm in}\quad {\rm K}_0(E)=\mathbb{Z}\left[\frac{1}{3}\right]\oplus \mathbb{Z}\left[\frac{1}{2}\right].
     $$
     Similarly,
     for any $x\in \mathbb{Z}\left[\frac{1}{2}\right]\cap[0,\infty)$ and $0\neq y\in \mathbb{Z}\left[\frac{1}{3}\right]\cap(0,\infty)$, we have
     $$\left(\begin{array}{c}
       y \\
       0
     \end{array}\right)\geq\left(\begin{array}{c}
       0 \\
       x
     \end{array}\right)\quad {\rm in}\quad {\rm K}_0(E).$$
So we have
$$
{\rm K}_0^+(E)=\{\left(\begin{array}{c}
       x \\
       y
     \end{array}\right)\in \mathbb{Z}\left[\frac{1}{3}\right]\oplus \mathbb{Z}\left[\frac{1}{2}\right]\mid x>0\,\,{\rm or} \,(x=0, y\geq 0)  \}.
$$
Then ${\rm K}_0(E)$ is unperforated.

Now we have the following  exact sequence is a full unital extension of simple AD algebras ($I$ is actually an AT algebra; see 
Proposition \ref{sat}) with trivial boundary maps:
$$
0\rightarrow I\rightarrow E\to E/I\rightarrow 0.
$$
The ${\rm K}_0$-group extension is pure exact, but the ${\rm K}_1$-group extension
$$
0\to \mathbb{Z}\xrightarrow{\times 2} \mathbb{Z}\to \mathbb{Z}_2\to 0
$$
is not pure exact. This means that $E$ is not K-pure.   By \cite[Theorem 5.3]{ALcounter}, ${\rm K}_*(E)$ is not weakly unperforated. More concretely,
$$
\left(\left(\begin{array}{c}
       0 \\
       1
     \end{array}\right),2\right)\in {\rm K}_0(E)\oplus {\rm K}_1(E)=(\mathbb{Z}\left[\frac{1}{3}\right]\oplus \mathbb{Z}\left[\frac{1}{2}\right])\oplus \mathbb{Z}
$$
is a positive element, but ${\scriptsize
\left(\left(\begin{array}{c}
       0 \\
       1/2
     \end{array}\right),1\right)}
$ is not positive.

Since any A$\mathcal{HD}$ algebra of real rank zero is K-pure (see \cite[Proposition 4.4]{DE}), then $E$ is not an A$\mathcal{HD}$ algebra.
\end{proof}
\begin{corollary}
  There exists a real rank zero inductive limit of 1-NCCW complexes, which has real rank zero and torsion free ${\rm K}_1$ but is not an AT algebra.
\end{corollary}

\begin{proposition}\label{sat}
  The ideal $I$ constructed in \ref{ex1} is a stable AT algebra of real rank zero.
\end{proposition}
\begin{proof}
Firstly, we point out that $I$ is an AT algebra.
This just comes from the following commutative diagram:
$$
\xymatrixcolsep{3pc}
\xymatrix{
{\,\,M_{3}(C(\mathbb{T}))\,\,} \ar[d]_-{\rho_1} \ar[r]^-{\gamma_{1,3}}
& {\,\,M_{3^3}(C(\mathbb{T}))\,\,} \ar[d]_-{\rho_2} \ar[r]^-{\gamma_{3,5}}
& {\,\,M_{3^5}(C(\mathbb{T}))\,\,} \ar[r]^-{} \ar[d]_-{\rho_3}
& {\,\,\cdots\,\,}
 \\
{\,\,I_1\,\,} \ar[r]_-{\psi_{1,3}} \ar[ru]^-{\varrho_1}
& {\,\,I_3 \,\,} \ar[r]_-{\psi_{3,5}} \ar[ru]^-{\varrho_2}
& {\,\,I_5 \,\,} \ar[r]_-{}\ar[ru]^-{\varrho_3}
& {\,\,\cdots\,\,}}
$$
where for each $n\geq 1$, $\rho_n$ is some natural embedding map,
$$
\gamma_{2n-1,2n+1}(g)=\left(\begin{array}{cccc}
         g(e^{2\pi {\rm i}t}) &  & & \\
         & g(e^{2\pi {\rm i}t_{2n-1}}) & & \\
         & & g(e^{2\pi {\rm i}t_{2n}}) & \\
         & & & g(e^{2\pi {\rm i}t_{2n-1}})  \\
      \end{array}\right)
$$
for any $g\in M_{3^{2n-1}}(C(\mathbb{T}))$,
$$
\varrho_n(f)=h,\quad h\in M_{3^{2n+1}}(C(\mathbb{T}))
$$
with
$$
h(e^{2\pi {\rm i}t})=u_n\left(\begin{array}{cccc}
         f(t) &  & & \\
         & f(t_{2n-1}) & & \\
         & & f(t_{2n}) & \\
         & & & f(t_{2n-1})  \\
      \end{array}\right)u_n^*
$$
for any $f\in I_{2n-1}$ and for some $u_n\in M_{3^{2n+1}}(C(\mathbb{T}))$.

Thus, we have $I=\lim\limits_{\longrightarrow}(M_{3^{2n-1}}(C(\mathbb{T})),\gamma_{2n-1,2n+1})$
and it is easily seen that the scale of ${\rm K}_0(I)$ is exactly ${\rm K}_0^+(I)$. By \cite[Proposition 3.1]{R1997}, $I$ is stable.
\end{proof}

\section{Obstruction in quotients}
In this section, we construct two unital, separable, nuclear, non-K-pure ASH algebras of stable rank one and  real rank zero with the same scaled  total Cuntz semigroup (hence, the same ordered scaled total K-theory), and we show that they are not isomorphic.
\begin{notion}\label{b1b2}\rm
Recall the ${\rm C}^*$-algebras 
 constructed in \cite[Theorem 3.3]{DL3}.
Choose $\{t_n\}$  as in \ref{ex1}, then $\{t_n\}$ is dense in $(0,1)$.
For any fixed $n\in \mathbb{N}$, $${\rm Prim}(M_{3^{n}}(\widetilde{\mathbb{I}}_2))=\{\theta_1,\theta_2\}\cup(0,1).$$
Define $\varphi_n: M_{3^{n+1}}(\widetilde{\mathbb{I}}_2)\rightarrow M_{3^{n+2}}(\widetilde{\mathbb{I}}_2)$ by
 $$
 \varphi_n(f,a)
 =
 u_n\left(\begin{array}{ccc}
                 f(t) &   &   \\
                 & f(t_n) &  \\
                 &  & f(t_n)
               \end{array}
\right)u_n^*.
 $$
 Set
$$
D_n= M_{3^n}\oplus\cdots\oplus M_3\oplus \mathbb{C}\oplus M_3\oplus \cdots \oplus M_{3^n}.
$$
Define the connecting maps
$$
\phi_{n,n+1}, \phi'_{n,n+1}:M_{3^{n+1}}(\widetilde{\mathbb{I}}_2)\oplus D_n   \rightarrow M_{3^{n+2}}(\widetilde{\mathbb{I}}_2)\oplus D_{n+1}
$$ as follows:
$$
\phi_{n,n+1}((f,a),(x_{-n},\cdots ,x_n))=
\varphi_n(f,a)\oplus (a(\theta_{1}),x_{-n},\cdots ,x_n,a(\theta_{1})),
$$
$$
\phi'_{n,n+1}((f,a),(x_{-n},\cdots ,x_n))=
\varphi_n(f,a)\oplus (a(\theta_{1}),x_{-n},\cdots ,x_n,a(\theta_{2})),
$$
where $(f,a)\in M_{3^{n+1}}(\widetilde{\mathbb{I}}_2)$ and  
$(x_{-n},\cdots ,x_n)\in D_n$.

Write
$$A_1=\lim\limits_{\longrightarrow}(M_{3^{n+1}}(\widetilde{\mathbb{I}}_2)\oplus D_n ,\phi_{n,n+1}),\,\,
A_2=\lim\limits_{\longrightarrow}(M_{3^{n+1}}(\widetilde{\mathbb{I}}_2)\oplus D_n ,\phi'_{n,n+1}).$$
 Then $A_1,A_2$ have real rank zero and stable rank one.

It is proved in
\cite[Theorem 3.3]{DL3}  that $A_1$, $A_2$ are not isomorphic and
$$
(\mathrm{K}_*(A_1),\mathrm{K}_*^+(A_1))
\cong(\mathrm{K}_*(A_2),\mathrm{K}_*^+(A_2)).
$$
By \cite[Theorem 9.1]{DG}, we have
$$
(\underline{\mathrm{K}}(A_1),\underline{\mathrm{K}}(A_1)_+)\ncong
(\underline{\mathrm{K}}(A_2),\underline{\mathrm{K}}(A_2)_+),
$$
where the orders for $\underline{\mathrm{K}}(A_1)$ and $\underline{\mathrm{K}}(A_2)$ are both the Dadarlat-Gong order. 
\end{notion}

\begin{notion}\rm
Denote by $\mathbf{Z}$ the subgroup of $\mathbb{Z}[\frac{1}{3}]\oplus \prod_{\mathbb{Z}}\mathbb{Z}$ consisting of all the elements of the form $(a, \prod_{\mathbb{Z}} a_m)$ with $a_m=3^{|m|} \cdot a$ for all large enough $|m|$.

For $A_1$, $A_2$, by \cite[Theorem 3.3]{DL3}, we have
$$
\mathrm{K}_0(A_i)= \mathbf{Z},\quad  \mathrm{K}_1(A_i)=\mathbb{Z}_2,
$$
and
$$
 \mathrm{K}_0^+(A_i)=\mathbf{Z}_+:=\mathbf{Z}\cap ( \mathbb{R}_+ \oplus \prod_{\mathbb{Z}}\mathbb{R}_+),
$$
where $i=1,2$ and $\mathbb{R}_+=[0,+\infty)$.
\end{notion}

\begin{example}\label{ex2}
Now we construct two $C^*$-algebras which are real rank zero inductive limits of 1-NCCW complexes with the same total K-theory but are not isomorphic.

Recall the construction  in Theorem \ref{ex1}, we already have $\{t_n\}, C_n$, $\psi_{n,n+1}:C_n \rightarrow C_{n+1}$ and
$$
{\rm Prim}(C_n)=\{\theta_{1}, \theta_{2},\theta_{3}\}\cup\coprod_{i=1}^{2}(0,1)_i.$$
We also have $D_n$
as in \ref{b1b2}.


Define the connecting maps
$$\rho_{n,n+1}, \rho'_{n,n+1}: C_{n+1}\oplus D_n\rightarrow C_{n+2}\oplus D_{n+1}
$$
 by
$$
\rho_{n,n+1}((f,a),(x_{-n},\cdots ,x_n))
=
\psi_{n+1,n+2}(f,a)\oplus (a(\theta_{1}),x_{-n},\cdots ,x_n,a(\theta_{1})),
$$
$$
\rho'_{n,n+1}((f,a),(x_{-n},\cdots ,x_n))
=
\psi_{n+1,n+2}(f,a)\oplus (a(\theta_{1}),x_{-n},\cdots ,x_n,a(\theta_{2})),
$$
where $(f,a)\in C_{n+1}$ and $(x_{-n},\cdots ,x_n)\in D_n$.

Write
$$E_1=\lim\limits_{\longrightarrow}(C_{n+1}\oplus D_n,\rho_{n,n+1}),\,\,\,\, E_2=\lim\limits_{\longrightarrow}(C_{n+1}\oplus D_n,\rho'_{n,n+1}).$$
Then $E_1,E_2$ are unital $C^*$-algebras of real rank zero and stable rank one.

For each $C_n$, since we have the homotopy of homomorphisms as
$$
\theta_1\oplus \theta_3=0_2\sim_h 1_2=\theta_2\oplus \theta_3,
$$
we obtain
$${\rm KK}(\theta_1)={\rm KK}(\theta_2).$$
This implies
$${\rm KK}(\rho_{n,n+1})={\rm KK}(\rho'_{n,n+1}),$$ then
by the universal multi-coefficient theorem in \cite{DL2}, we have
 $$
(\underline{\mathrm{K}}(E_1),\underline{\mathrm{K}}(E_1)_+,[1_{E_1}])\cong
(\underline{\mathrm{K}}(E_2),\underline{\mathrm{K}}(E_2)_+,[1_{E_2}]).
$$

Note that the algebra $I$ in Theorem \ref{ex1} 
 sits as an ideal of $E_1$ and $E_2$, respectively. Then there are  extensions
$$
0\to I\rightarrow E_i\rightarrow{}A_i\to 0, \quad i=1,2,
$$
where $A_1,A_2$ are exactly the algebras in \ref{b1b2}.
Now $A_1\ncong A_2$ implies $E_1\ncong E_2$.
\end{example}

\begin{remark} 
Here, we give a description for Example \ref{ex2} from the point view of extension. From  Theorem \ref{ex1} and  Proposition \ref{sat}, we have
 $$\mathrm{K}_0(I)=\mathbb{Z}\left[\frac{1}{2}\right],\quad
 \mathrm{K}_0^+(I)=\mathbb{Z}\left[\frac{1}{2}\right]\cap[0,\infty)\quad {\rm and} \quad \mathrm{K}_1(I)=\mathbb{Z}.$$

From the inductive construction above, for each $i=1,2,$ the six-term exact sequence of
$$
0\to I\rightarrow E_i\rightarrow{}A_i\to 0
$$
splits into the following two short exact sequences (unital version)
$$
0\to \mathrm{K}_0(I)\xrightarrow{\zeta_i} (\mathbb{Z}\left[\frac{1}{2}\right]\oplus \mathbf{Z},(0,(1, \prod_{\mathbb{Z}} 3^{|m|})))\xrightarrow{\eta_i} (\mathrm{K}_0(A_i),(1, \prod_{\mathbb{Z}} 3^{|m|}))\to 0,
$$
$$
0\to \mathrm{K}_1(I)\xrightarrow{\times2}\mathbb{Z}\xrightarrow{\nu_i} \mathrm{K}_1(A_i)\to 0,
$$
where $\zeta_i(a)=(a, 0)$, 
and
$\eta_i(x, (b, \prod_\mathbb{Z}b_m))=(b, \prod_\mathbb{Z}b_m).$
Since the $K_1$-sequence isn't a pure group extension, we have $E_i$ is not K-pure.
Moreover,  $(\mathrm{K}_*(E_i),\mathrm{K}_*^+(E_i))$ is torsion free and is not weakly unperforated.

Here, we point out that both the extensions
$$
0\to I\xrightarrow{\iota_i} E_i\xrightarrow{\pi_i}A_i\to 0, \quad i=1,2,
$$
are not full, as they are not stenotic (see \cite[Corollary 3.20]{ERR2}). 

\end{remark}

\begin{remark}\label{Cutotal}
From Example \ref{ex2},  we have a canonical isomorphism
 $$
\lambda:\,(\underline{\mathrm{K}}(E_1),\underline{\mathrm{K}}(E_1)_+,[1_{E_1}])\to
(\underline{\mathrm{K}}(E_2),\underline{\mathrm{K}}(E_2)_+,[1_{E_2}]).
$$
At once, we obtain the following commutative diagram in $\Lambda$-category
$$
\xymatrixcolsep{2pc}
\xymatrix{
{\,\,\underline{\mathrm{K}}(I)\,\,} \ar[r]^-{\underline{\mathrm{K}}(\iota_1)}\ar[d]_-{id}
& {\,\,\underline{\mathrm{K}}(E_1)\,\,} \ar[d]_-{ \lambda}
 \\
{\,\,\underline{\mathrm{K}}(I) \,\,}\ar[r]_-{\underline{\mathrm{K}}(\iota_2)}
& {\,\,\underline{\mathrm{K}}(E_2)\,\,}.}
$$
Since $A_1,A_2$ and $I$ are K-pure, it is straightforward to check that $E_1$ and $E_2$ has the same total Cuntz semigroup (\cite{AL,ALcounter}). This means that the ``ideal version of ordered scaled total K-theory'' does not distinguish $E_1,E_2$,
and hence, the information of ``total K-theory in the quotient'' is also necessary for the classification! We are able to give a classification in terms of the total K-theory of ideals and quotients; see Theorem \ref{mainthm}.
\end{remark}
\begin{remark}
We point out that the algebras $A_1,A_2$ in \ref{b1b2} have the ideals isomorphic to $M_{3^n}$, so $A_1,A_2$ are not $\mathcal{Z}$-stable, which implies $E_1,E_2$ in Example \ref{ex2} are not $\mathcal{Z}$-stable (\cite[ 4.3]{TW}). There also exist real rank zero $\mathcal{Z}$-stable $C^*$-algebras with the same total K-theory but they are not isomorphic, one can check the algebras $E_1,E_2$ in \cite[Example 6.2]{ALbock}. We will construct such examples in the class of real rank zero limits of
1-NCCW complexes. We just need to make a slight modification in the inductive structure.
\end{remark}
Let us use $a^{\thicksim k}$ to denote 
${\rm diag}\{a,\cdots,a\}$, where $a$ repeats $k$ times.
\begin{example}\label{ex Z}
Choose $\{t_n\}, C_n, \psi_{n,n+1}:C_n \rightarrow C_{n+1}$ as in Example \ref{ex2}. Define
$$
D_n'=\underbrace{M_{3^n}\oplus M_{3^n}\oplus \cdots \oplus M_{3^n}}_{2n+1\,\,copies}.
$$
Define the connecting maps $\varrho_{n,n+1}, \varrho'_{n,n+1}: C_{n+1}\oplus D_n'\rightarrow C_{n+2}\oplus D_{n+1}'$ by
$$
\varrho_{n,n+1}((f,a),(x_{-n},\cdots,x_n))
=
\psi_{n+1,n+2}(f,a)\oplus (a(\theta_{1}),x_{-n}^{\sim3}, \cdots,x_n^{\sim3},a(\theta_{1})),
$$
$$
\varrho'_{n,n+1}((f,a),(x_{-n}, \cdots,x_n))
=
\psi_{n+1,n+2}(f,a)\oplus (a(\theta_{1}),x_{-n}^{\sim3}, \cdots ,x_n^{\sim3},a(\theta_{2})),
$$
where $(f,a)\in C_{n+1}$, $(x_{-n}, \cdots,x_n)\in D_n'$. 

Write
$$
E_1'=\lim\limits_{\longrightarrow}(C_{n+1}\oplus D_n',\varrho_{n,n+1}),\,\, E_2'=\lim\limits_{\longrightarrow}(C_{n+1}\oplus D_n',\varrho'_{n,n+1}).$$ Then $E_1',E_2'$ are unital $C^*$-algebras of real rank zero and stable rank one and
the algebra $I$ in Theorem \ref{ex1} also sits as an ideal of $E_1'$ and $E_2'$, respectively. Then we have extensions
$$
0\to I\rightarrow E_i'\rightarrow{}A_i'\to 0, \quad i=1,2.
$$
Note that 
each matrix algebra $M_{3^n}$ which sits as an ideal of $A_1$,$A_2$, respectively, is replaced by the UHF algebra $M_{3^\infty}$ for $A_1'$,$A_2'$, respectively.
Since the canonical unital embedding map from $M_{3^n}$ to $M_{3^\infty}$ induces the identity map from $\mathrm{K}_0(M_{3^n};\mathbb{Z}_2)(\cong\mathbb{Z}_2)$ to $\mathrm{K}_0(M_{3^\infty};\mathbb{Z}_2)(\cong\mathbb{Z}_2)$, the same argument
in the proof of \cite[Theorem 3.3]{DL3} works for
$$\mathrm{K}_0(A_i';\mathbb{Z}\oplus\mathbb{Z}_2)=\mathbb{Z}[\frac{1}{3}]\oplus \mathbb{Z}_2\oplus\mathbb{Z}_2\oplus
\prod_{j=-\infty}^{+\infty}\mathbb{Z}[\frac{1}{3}]\oplus \prod_{j=-\infty}^{+\infty}\mathbb{Z}_2.$$
Since there doesn't exist an order-preserving isomorphism between  $\mathrm{K}_0(A_1';\mathbb{Z}\oplus\mathbb{Z}_2)$ and $\mathrm{K}_0(A_2';\mathbb{Z}\oplus\mathbb{Z}_2)$,
so we still have $A_1'\ncong A_2' $.
Thus, $E_1'\ncong E_2'$.

It also can  be checked that $A_i'\cong A_i\otimes M_{3^\infty}$, $i=1,2$,  then $A_1'$ and $A_2'$ are $\mathcal{Z}$-stable. Since $I$ is also $\mathcal{Z}$-stable, we conclude that $E_1',E_2'$ are $\mathcal{Z}$-stable (\cite[4.3]{TW}).
With a same argument for $E_1,E_2$,  we get
 $$
(\underline{\mathrm{K}}(E_1'),\underline{\mathrm{K}}(E_1')_+,[1_{E_1'}])\cong
(\underline{\mathrm{K}}(E_2'),\underline{\mathrm{K}}(E_2')_+,[1_{E_2'}])
$$
and  $E_1,E_2$ have the same total Cuntz semigroup.
\end{example}

\section{full extension case and classification}

The algebra $E$ in Theorem \ref{ex1} forms a unital  full extension.
Although $E_1,E_2$  in Example \ref{ex2} form unital extensions but both of them are not full. Thus, it may seem that the non-fullness could induce some obstructions for the classification, but we point out that this is not the case. That is, there exist two unital full extensions with similar phenomena. We propose the following modification.

\begin{example}
Set $l_1=9$ and for any $n\geq 1$,
take
$$
l_{n+1}=2\cdot l_n+3^{n+1}+2\cdot 4^{n-1}+(3+9+27+\cdots+3^n)\cdot4^n,
$$
thus, we obtain a sequence $\{l_n\}_{n=1}^\infty$ inductively.

Set
$$
F_{1,n}'= M_{3^n}(\mathbb{C})\oplus M_{3^n}(\mathbb{C})\oplus M_{l_n}(\mathbb{C}),\,\,\,\,
F_{2,n}'= M_{2\cdot3^n}(\mathbb{C})\oplus M_{l_n+3^n}(\mathbb{C}).
$$
Define  $\varphi_{0,n}',\,\varphi_{1,n}':\,F_{1,n}'\to F_{2,n}'$ as
$$
\varphi_{0,n}' (a)=
\left(\begin{array}{cc}
        a(\theta_{1}) &  \\
         & a(\theta_{1})
      \end{array}
\right)\oplus \left(\begin{array}{cc}
        a(\theta_{1}) &  \\
         & a(\theta_{3})
      \end{array}
\right),$$
$$
\varphi_{1,n}' (a)=
\left(\begin{array}{cc}
        a(\theta_{2}) &  \\
         & a(\theta_{2})
      \end{array}
\right)\oplus \left(\begin{array}{cc}
        a(\theta_{2}) &  \\
         & a(\theta_{3})
      \end{array}
\right).
$$
Set $C_n'=C_n'(F_{1,n}',F_{2,n}',\varphi_{0,n}',\varphi_{1,n}')$.
Then $${\rm Prim}(C_n')=\{\theta_{1}, \theta_{2},\theta_{3}\}\cup\coprod_{i=1}^{2}(0,1)_{i}.$$
For each $n\geq 1$
and $i=1,2$,
let $\{t_n\}$ be a dense sequence in $(0,1)$. Still take $$
D_n= M_{3^n}\oplus\cdots\oplus M_3\oplus \mathbb{C}\oplus M_3\oplus \cdots \oplus M_{3^n}.
$$


Define $\tau_{n,n+1}^{0}: C_{n+1}'\oplus D_n\rightarrow C_{n+2}'$ by
\begin{align*}
 \tau_{n,n+1}^0((f,a),(x_{-n},\cdots,x_n))=
  & u\left(\begin{array}{ccc}
                 f(t,1) &  &   \\
                 & f(t_n, 1) &  \\
                 & &  f(t_n, 1)  \\
               \end{array}
\right)u^* \\
   & \oplus v\left(\begin{array}{cccc}
                 f(t,2) &  & &  \\
                 & f(t_n, 1) & & \\
                 &     & \langle x_n \rangle^{\sim 2\cdot 4^{n-1}} & \\
                  & & & f(t_n, 2)  \\
      \end{array}
\right)v^*,
\end{align*}
where
 $(f,a)\in C_{n+1}'$, $(x_{-n},\cdots,x_n)\in D_n$, $\langle x_n \rangle= {\rm diag}\{x_{-n},\cdots,x_n\}$,  $u\in M_{2\cdot 3^{n+1}}(\mathbb{C})$ and $v\in M_{l_{n+1}+3^{n+1}}(\mathbb{C})$ are unitaries such that
\begin{align*}
  \pi_e\circ \tau^0_{n,n+1}((f,a),(x_{-n},\cdots,x_n))= & \left(\begin{array}{cc}
         a(\theta_1) &  \\
         & f(t_n,1)
      \end{array}
\right)\oplus \left(\begin{array}{cc}
         a(\theta_2) &  \\
         & f(t_n,1)
      \end{array}
\right) \\
   &\oplus  \left(\begin{array}{ccc}
         a(\theta_3) &  &\\
         &  \langle x_n \rangle^{\sim 2\cdot 4^{n-1}} & \\
         & & f(t_n,2)
      \end{array}
\right).
\end{align*}
Define the connecting maps
$$
\tau_{n,n+1}, \tau'_{n,n+1}: {C}_{n+1}'\oplus D_n\rightarrow {C}_{n+2}'\oplus D_{n+1}$$
 by
 \begin{align*}
  \tau_{n,n+1}((f,a),(x_{-n},\cdots,x_n) )
=& \tau_{n,n+1}^0((f,a),(x_{-n},\cdots,x_n))  \\
  &\oplus (a(\theta_1),x_{-n},\cdots,x_n,a(\theta_1)),
 \end{align*}
  \begin{align*}
\tau'_{n,n+1}((f,a),(x_{-n},\cdots,x_n) )
=& \tau_{n,n+1}^0((f,a),(x_{-n},\cdots,x_n)) \\
  &\oplus (a(\theta_1),x_{-n},\cdots,x_n,a(\theta_2)),
 \end{align*}
where $(f,a)\in C_{n+1}'$ and $(x_{-n},\cdots,x_n)\in D_n$.
Write
$$
E_1'=\lim\limits_{\longrightarrow}(C_{n+1}'\oplus D_n,\tau_{n,n+1}),\,\,
E_2'=\lim\limits_{\longrightarrow}(C_{n+1}'\oplus D_n,\tau'_{n,n+1}).
$$
Then $E_1',E_2'$ are of real rank zero and stable rank one and
 we have
$$
0\to I\to E_i'\to A_i\to 0,\quad i=1,2,
$$
are unital full extensions (the fullness can be checked through \cite[Corollary 3.20]{ERR2}), 
$E_1',E_2'$ satisfying all the statement for $E_1,E_2$ in the last section.
\end{example}

From the examples we built, we have shown that ASH algebras or  extensions of A$\mathcal{HD}$ algebras can not be classified by the ideal version of total K-theory. Hence,  total K-theory of the quotients are also necessary for the classification. We will give some evidences that the total  total K-theory of ideals and quotients are sufficient for classification. We list the following:
\begin{theorem}\label{mainthm}
Given A$\mathcal{HD}$ algebras $A_i,B_i,$ $i=1,2$  of stable rank one and real rank zero with $A_i$ unital and $B_i$ stable and
 unital full extensions with trivial boundary maps:
$$
0\to B_i\to E_i\to A_i\to 0,\quad i=1,2.
$$
If the following diagram is commutative in $\Lambda$-category
$$\xymatrixcolsep{2pc}
\xymatrix{
 {\,\,\underline{\mathrm{K}}(B_1)\,\,} \ar[d]_-{\cong_+}\ar[r]^-{\underline{\mathrm{K}}(\iota_1)}
& {\,\,(\underline{\mathrm{K}}(E_1),[1_{E_1}])\,\,} \ar[d]_-{\cong_+} \ar[r]^-{\underline{\mathrm{K}}(\pi_1)}
& {\,\,(\underline{\mathrm{K}}(A_1),[1_{A_1}])\,\,}\ar[d]_-{\cong_+}\\
{\,\,\underline{\mathrm{K}}(B_2)\,\,} \ar[r]_-{\underline{\mathrm{K}}(\iota_2)}
& {\,\,(\underline{\mathrm{K}}(E_2),[1_{E_2}]) \,\,} \ar[r]_-{\underline{\mathrm{K}}(\pi_2)}
& {\,\,(\underline{\mathrm{K}}(A_2),[1_{A_2}]) \,\,},}
$$
where $\cong_+$ means the ordered scaled isomorphism, then
 $E_1\cong E_2.$
\end{theorem}
\begin{proof} Denote the isomorphism maps by $\gamma_0,\gamma,\gamma_1$ in
$$\xymatrixcolsep{2pc}
\xymatrix{
 {\,\,\underline{\mathrm{K}}(B_1)\,\,} \ar[d]_-{\gamma_0}\ar[r]^-{\underline{\mathrm{K}}(\iota_1)}
& {\,\,(\underline{\mathrm{K}}(E_1),[1_{E_1}])\,\,} \ar[d]_-{\gamma} \ar[r]^-{\underline{\mathrm{K}}(\pi_1)}
& {\,\,(\underline{\mathrm{K}}(A_1),[1_{A_1}])\,\,}\ar[d]_-{\gamma_1}\\
{\,\,\underline{\mathrm{K}}(B_2)\,\,} \ar[r]_-{\underline{\mathrm{K}}(\iota_2)}
& {\,\,(\underline{\mathrm{K}}(E_2),[1_{E_2}]) \,\,} \ar[r]_-{\underline{\mathrm{K}}(\pi_2)}
& {\,\,(\underline{\mathrm{K}}(A_2),[1_{A_2}]) \,\,}.}
$$
Since for $ i=1,2,$ $B_i$, $A_i$ are A$\mathcal{HD}$ algebras of real rank zero, by \cite[Theorem 9.1]{DG}, there exist isomorphisms
$\phi_0:\,B_1\to B_2$
and
$\phi_1:\,A_1\to A_2$
 such that $ \underline{\mathrm{K}}(\phi_j)=\gamma_j$, $j=0,1$. By assumption,
we have the following diagram: 
$$\xymatrixcolsep{2pc}
\xymatrix{
 {\,\,\underline{\mathrm{K}}(B_1)\,\,} \ar[d]_-{\underline{\mathrm{K}}(\phi_0)}\ar[r]^-{\underline{\mathrm{K}}(\iota_1)}
& {\,\,(\underline{\mathrm{K}}(E_1),[1_{E_1}])\,\,} \ar[d]_-{\gamma} \ar[r]^-{\underline{\mathrm{K}}(\pi_1)}
& {\,\,(\underline{\mathrm{K}}(A_1),[1_{A_1}])\,\,}\ar[d]_-{\underline{\mathrm{K}}(\phi_1)}\\
{\,\,\underline{\mathrm{K}}(B_2)\,\,} \ar[r]_-{\underline{\mathrm{K}}(\iota_2)}
& {\,\,(\underline{\mathrm{K}}(E_2),[1_{E_2}]) \,\,} \ar[r]_-{\underline{\mathrm{K}}(\pi_2)}
& {\,\,(\underline{\mathrm{K}}(A_2),[1_{A_2}]) \,\,} .}
$$
Restricting to $\mathrm{K}_*$, the following diagram is commutative with exact rows:
$$
\xymatrixcolsep{2pc}
\xymatrix{
{\,\,0\,\,} \ar[r]^-{}
& {\,\,\mathrm{K}_*(B_1)\,\,} \ar[d]_-{\mathrm{K}_*(\phi_0)} \ar[r]^-{\mathrm{K}_*(\iota_1)}
& {\,\,(\mathrm{K}_*(E_1),[1_{E_1}])\,\,} \ar[d]_-{\gamma_*} \ar[r]^-{\mathrm{K}_*(\pi_1)}
& {\,\,(\mathrm{K}_*(A_1),[1_{A_1}])\,\,} \ar[d]_-{\mathrm{K}_*(\phi_1)} \ar[r]^-{}
& {\,\,0\,\,} \\
{\,\,0\,\,} \ar[r]^-{}
& {\,\,\mathrm{K}_*(B_2)\,\,} \ar[r]_-{  \mathrm{K}_*(\iota_2)}
& {\,\,(\mathrm{K}_*(E_2),[1_{E_2}]) \,\,} \ar[r]_-{\mathrm{K}_*(\pi_2)}
& {\,\,(\mathrm{K}_*(A_2),[1_{A_2}]) \,\,} \ar[r]_-{}
& {\,\,0\,\,},}
$$ 
where $\mathrm{K}_*(\phi_j)$ is induced by $\phi_j$, $j=0,1,$
$\gamma_*$ is the restriction of $\gamma$.

Now we have the following commutative diagram with exact rows:
$$
\xymatrixcolsep{2pc}
\xymatrix{
{\,\,0\,\,} \ar[r]^-{}
& {\,\,\mathrm{K}_*(B_1)\,\,} \ar[d]_-{{\rm id}} \ar[r]^-{\mathrm{K}_*(\iota_1)}
& {\,\,(\mathrm{K}_*(E_1),[1_{E_1}])\,\,} \ar[d]_-{\eta_*} \ar[r]^-{\mathrm{K}_*(\phi_1\circ \pi_1) }
& {\,\,(\mathrm{K}_*(A_2),[1_{A_2}])\,\,} \ar[d]_-{{\rm id}} \ar[r]^-{}
& {\,\,0\,\,\,} \\
{\,\,0\,\,} \ar[r]^-{}
& {\,\,\mathrm{K}_*(B_1)\,\,} \ar[r]_-{ \mathrm{K}_*( \iota_2\circ\phi_0)}
& {\,\,(\mathrm{K}_*(E_2),[1_{E_2}]) \,\,} \ar[r]_-{\mathrm{K}_*(\pi_2)}
& {\,\,(\mathrm{K}_*(A_2),[1_{A_2}]) \,\,} \ar[r]_-{}
& {\,\,0\,\,}.}
$$

Then by Theorem \ref{strong wei} (the naturality comes from \cite[Theorem 4.14]{GR}), Proposition \ref{AL CFP} and Proposition \ref{AL full},  the two extensions
$$
0\to B_1 \xrightarrow{\iota_1} E_1\xrightarrow{\phi_1\circ \pi_1} A_2\to 0
$$
and
$$
0\to B_1 \xrightarrow{\iota_2\circ \phi_0} E_2\xrightarrow{\pi_2} A_2\to 0
$$
are strongly unitarily equivalent,   and consequently $E_1\cong E_2$.

\end{proof}
\begin{remark}
In the above proof, we require $A_1,A_2,B_1,B_2$ are A$\mathcal{HD}$ algebras of real rank zero. As we only use the fact that these algebras can be completely classified by the total K-theory, so this assumption can be relaxed to more general class of $C^*$-algebras.
\end{remark}

\section{The example with bounded torsion}

Now we present a real rank zero inductive limit of 1-NCCW complexes which has torsion free ${\rm K}_0$ and non trivial, but bounded torsion ${\rm K}_1$ but is not an AD algebra.
This is in contrast with Theorem \ref{ex1}.
\begin{example} \label{ex torsion}
Set
$$
F_{1}=\mathbb{C}\oplus M_{2}(\mathbb{C})\oplus \mathbb{C}\oplus M_{2}(\mathbb{C}),\,\,
F_{2}=M_{4}(\mathbb{C})\oplus M_{4}(\mathbb{C}).
$$
Define  $\varphi_{0},\,\varphi_{1}:\,F_{1}\to F_{2}$ as
$$
\varphi_0 (a)=
\left(\begin{array}{cccc}
        a(\theta_{1}) & & & \\
         & a(\theta_{1}) & & \\
         & & a(\theta_{1}) & \\
         & & & a(\theta_{1})
      \end{array}
\right)\oplus \left(\begin{array}{ccc}
        a(\theta_{2}) & &  \\
         & a(\theta_{3}) &  \\
         & & a(\theta_{3})
      \end{array}
\right),$$
$$
\varphi_1 (a)=
\left(\begin{array}{cc}
        a(\theta_{2}) &  \\
         & a(\theta_{2})
      \end{array}
\right)\oplus \left(\begin{array}{cc}
        a(\theta_{4}) &  \\
         & a(\theta_{4})
      \end{array}
\right).
$$
Set $F=F(F_{1},F_{2},\varphi_{0},\varphi_{1})$.
Then
$${\rm Prim}(F)=\{\theta_{1}, \theta_{2},\theta_{3},\theta_{4}\}\cup\coprod_{i=1}^{2}(0,1)_{i}.$$

For any $n\in \mathbb{N}$, set
$
E_n=M_{5^{n-1}}(F).
$ Then
$$
{\rm K}_0(E_n)\cong {\rm Ker}\left(\begin{array}{cccc}
                          4 & -2 & 0 & 0 \\
                          0 & 1 & 2 & -2
                        \end{array}
\right)  \cong \mathbb{Z}\oplus \mathbb{Z},
$$ where $ (0,0,1,1)^{\rm T}$ and $(1,2,0,1)^{\rm T}$ in
${\scriptsize {\rm Ker}\left(\begin{array}{cccc}
                          4 & -2 & 0 & 0 \\
                          0 & 1 & 2 & -2
                        \end{array}\right)  }$ correspond to the generator $(0,1)$ and $(1,0)$ in $\mathbb{Z}\oplus \mathbb{Z}$, respectively.

Define $\phi_{n,n+1}: E_n\rightarrow E_{n+1}$ by
\begin{align*}
\phi_{n,n+1}(f,a)= & u\left(\begin{array}{cc}
                 f(t,1) &  \\
                 & f(t_n,1)^{\sim 4}   \\
               \end{array}
\right)u^* \\
   & \oplus v\left(\begin{array}{ccccc}
                 f(t,2) &  & \\
                 & f(t_n,1)^{\sim 2}  & \\
                 &     & f(t_n,2)^{\sim 2}  \\
      \end{array}
\right)v^*
\end{align*}
 where
 $(f,a)\in {E}_n$,
 $u,v$ are unitaries in $M_{4\cdot 5^n}(\mathbb{C})$ 
 such that
\begin{align*}
\pi_e\circ \phi_{n,n+1}(f,a)= & \left(\begin{array}{cc}
         a(\theta_1)  &\\
         &   f(t_n,1)  \\
      \end{array}
\right)\oplus \left(\begin{array}{cc}
         a(\theta_2) &  \\
         &   f(t_n,1)^{\sim 2}  \\
      \end{array}
\right) \\
 & \oplus \left(\begin{array}{cc}
         a(\theta_3) & \\
         &   f(t_n,2) \\
      \end{array}
\right)\oplus \left(\begin{array}{ccc}
         a(\theta_4) & & \\
         &   f(t_n,1)&  \\
          && f(t_n,2) \\
      \end{array}
\right).
\end{align*}

Then $E=\lim\limits_{\longrightarrow}(E_n,\phi_{n,n+1})$ is a unital ASH algebra of real rank zero and stable rank one with
$$
{\rm K}_0(E)=\mathbb{Z}\left[\frac{1}{3}\right]\oplus \mathbb{Z}\left[\frac{1}{5}\right]
\quad{\rm
and}\quad
{\rm K}_1(E)=\mathbb{Z}_4.
$$

Take $p_n$ to be the projection in $E_n$ with
$$
[p_n]=(0,0,1,1)^{\rm T}\in {\rm K}_0(E_n),
$$
and denote $B_n$ the ideal of $E_n$ generated by $p_n$. Then
$B=\lim\limits_{\longrightarrow}(B_n,\phi_{n,n+1})$ forms a unique proper ideal of $E$. With an almost same proof of Proposition \ref{sat}, we can obtain $B$ is a stable AD algebra of real rank zero. Denote $A= E/B=\lim\limits_{\longrightarrow}(E_n/B_n,\phi_{n,n+1}^*)$,
where $\phi_{n,n+1}^*$ is induced by $\phi_{n,n+1}$. $A$ is unital and of real rank zero.
Take
$$F'=\{f\in M_4(C[0,1])\mid f(0)=\lambda^{\sim4},~f(1)=\mu^{\sim2}~~ ,{\rm }~~\lambda\in\mathbb{C},~ \mu\in M_2\}.$$
Then $F'$ is in fact a quotient of $F$ and
$$E_n/B_n\cong M_{5^{n-1}}(F').$$
By \cite[Theorem 3.1]{ALajs}, we have $A$ is a AD algebra.

Note that
$$
{\rm K}_0(B)=\mathbb{Z}\left[\frac{1}{3}\right]
\quad{\rm
and}\quad
{\rm K}_1(B)=\mathbb{Z}_2,
$$
and
$$
{\rm K}_0(A)=\mathbb{Z}\left[\frac{1}{5}\right]
\quad{\rm
and}\quad
{\rm K}_1(A)=\mathbb{Z}_2.
$$
Since we have the ${\rm K}_1$-exact sequence
$$
0\to {\rm K}_1(B)\to {\rm K}_1(E)\to {\rm K}_1(A)\to 0
$$
as
$$
0\to \mathbb{Z}_2\to \mathbb{Z}_4\to \mathbb{Z}_2\to 0,
$$
which is not a pure group extension,
the extension
$$
0\to B\to E\to A\to 0
$$
is a non-K-pure extension, though it is unital full. 
Hence, $E$ is a non-K-pure ASH algebra, so $E$ is not an AD algebra.
\end{example}
\begin{remark}
 This is contrast with \cite[Theorem 4.4]{DL0} and one can also modify the above construction to show the $E$ in \cite[Example 4.5]{DL0} is also an ASH algebra (Dadarlat-Loring showed their $E$ is not an A$\mathcal{HD}$ algebra).
\end{remark}

\section*{Acknowledgements}


\end{document}